\documentclass{amsart}
\usepackage{graphicx}
\usepackage{amscd}
\usepackage{amsmath}
\usepackage{amsxtra}
\usepackage{amsfonts}
\usepackage{amssymb}

\newtheorem{theorem}{Theorem}[section]

\newtheorem{lemma}[theorem]{Lemma}
\newtheorem{proposition}[theorem]{Proposition}

\theoremstyle{definition}

\newtheorem{remark}[theorem]{Remark}

\theoremstyle{remark}

\renewcommand{\theclaim}{\textup{\theclaim}}

\numberwithin{equation}{section}

\def\openone

{\mathchoice

{\hbox{\upshape \small1\kern-3.3pt\normalsize1}}

{\hbox{\upshape \small1\kern-3.3pt\normalsize1}}

{\hbox{\upshape \tiny1\kern-2.3pt\SMALL1}}

{\hbox{\upshape \Tiny1\kern-2pt\tiny1}}}

\makeatletter

\newbox\ipbox

\newcommand{\ip}[2]{\left\langle #1\, , \,#2\right\rangle}
\newcommand{\diracb}[1]{\left\langle #1\mathrel{\mathchoice

{\setbox\ipbox=\hbox{$\displaystyle \left\langle\mathstrut
#1\right.$}

\vrule height\ht\ipbox width0.25pt depth\dp\ipbox}

{\setbox\ipbox=\hbox{$\textstyle \left\langle\mathstrut
#1\right.$}

\vrule height\ht\ipbox width0.25pt depth\dp\ipbox}

{\setbox\ipbox=\hbox{$\scriptstyle \left\langle\mathstrut
#1\right.$}

\vrule height\ht\ipbox width0.25pt depth\dp\ipbox}

{\setbox\ipbox=\hbox{$\scriptscriptstyle \left\langle\mathstrut
#1\right.$}

\vrule height\ht\ipbox width0.25pt depth\dp\ipbox}

}\right. }

\newcommand{\dirack}[1]{\left. \mathrel{\mathchoice

{\setbox\ipbox=\hbox{$\displaystyle \left.\mathstrut
#1\right\rangle$}

\vrule height\ht\ipbox width0.25pt depth\dp\ipbox}

{\setbox\ipbox=\hbox{$\textstyle \left.\mathstrut
#1\right\rangle$}

\vrule height\ht\ipbox width0.25pt depth\dp\ipbox}

{\setbox\ipbox=\hbox{$\scriptstyle \left.\mathstrut
#1\right\rangle$}

\vrule height\ht\ipbox width0.25pt depth\dp\ipbox}

{\setbox\ipbox=\hbox{$\scriptscriptstyle \left.\mathstrut
#1\right\rangle$}

\vrule height\ht\ipbox width0.25pt depth\dp\ipbox}

} #1\right\rangle}

\newcommand{\beq}{\begin{equation}}

\newcommand{\eeq}{\end{equation}}

\newcommand{\ora}{\overrightarrow}

\newcommand{\cj}[1]{\overline{#1}}

\newcommand{\M}{\mathcal{M}}

\newcommand{\br}{\mathbb{R}}
\newcommand{\bc}{\mathbb{C}}

\newcommand{\bn}{\mathbb{N}}

\def\blfootnote{\xdef\@thefnmark{}\@footnotetext}

\renewcommand{\mod}{\operatorname{mod}}

\input xy
\xyoption{all}
\usepackage{amssymb}

\newcommand{\norm}[1]{\lvert \lvert#1\rvert \lvert }

\def\-{^{-1}}

\def\ty{\emptyset}

\begin{document}

\title[Parseval frames of piecewise constant functions]{Parseval frames of piecewise constant functions}
\author{Dorin Ervin Dutkay}

\address{[Dorin Ervin Dutkay] University of Central Florida\\
	Department of Mathematics\\
	4000 Central Florida Blvd.\\
	P.O. Box 161364\\
	Orlando, FL 32816-1364\\
U.S.A.\\} \email{Dorin.Dutkay@ucf.edu}

\author{Rajitha Ranasinghe}

\address{[Rajitha Ranasinghe] University of Central Florida\\
	Department of Mathematics\\
	4000 Central Florida Blvd.\\
	P.O. Box 161364\\
	Orlando, FL 32816-1364\\
U.S.A.\\} \email{rajitha13@knights.ucf.edu }

\thanks{} 

\begin{abstract}
We present a way to construct Parseval frames of piecewise constant functions for $L^2[0,1]$. The construction is similar to the generalized Walsh bases. It is based on iteration of operators that satisfy a Cuntz-type relation, but without the isometry property. We also show how the Parseval frame can be dilated to an orthonormal basis and the operators can be dilated to true Cuntz isometries. 
 
\end{abstract}
\maketitle \tableofcontents

\section{Introduction}
In \cite{DPS14}, Dutkay et al. introduced a method of constructing orthonormal bases from representations of Cuntz algebras. Recall that the {\it Cuntz algebra} $\mathcal O_N$, where $N$ is an integer, $N\geq 2$, is the $C^*$-algebra generated by $N$ isometries $(S_i)_{i=0,\dots, N-1}$ on some Hilbert space $\mathcal H$ which satisfy the {\it Cuntz relations} 
\begin{equation}
S_i^*S_j=\delta_{ij}I_{\mathcal H}, \quad (i,j\in\{0,\dots, N-1\}),\quad \sum_{i=0}^{N-1}S_iS_i^*=I_{\mathcal H}.
\label{eqcuntz}
\end{equation} 

The basic idea was to start with some vector $v_0$ in $\mathcal H$ which is fixed by the first isometry, $S_0v_0=v_0$, and then apply all the Cuntz isometries $S_{\omega_1}\dots S_{\omega_n}v_0$ where $\omega_1,\dots,\omega_n\in\{0,\dots,N-1\}$. Eliminating the repetitions generated by the fact that $S_0v_0=v_0$, one can see immediately that the resulting family of vectors is orthonormal. The more delicate issue is, of course, its completeness. 

A particular case of this construction yields the classical Walsh basis on $L^2[0,1]$ and some variations of that yield generalized Walsh bases for $L^2[0,1]$ consisting of piece-wise constant functions, see \cite{DPS14, DP14}. 

In this paper, we will follow similar ideas, but with some important modifications. We will begin not with a Cuntz algebra representation, but with one where only the relation 
\begin{equation}
\sum_{i=0}^{N-1}\tilde S_i\tilde S_i^*=I_{\mathcal H}
\label{eqhalf}
\end{equation}
is satisfied. Again we will have a vector $v_0$ (in our case, the constant function $\mathbf 1$) with $\tilde S_0v_0=v_0$, and, by iterating the operators $\tilde S_i$, we will obtain a family $\tilde S_{\omega_1}\dots \tilde S_{\omega_n}v_0$ which is a Parseval frame. 

Recall that a {\it Parseval frame} for a Hilbert space $\mathcal H$ is a family of vectors $\{\tilde e_j : j\in J\}$ such that 
\begin{equation}
\|f\|^2=\sum_{j\in J}|\ip{f}{\tilde e_j}|^2,\quad (f\in\mathcal H).
\label{eqparseval}
\end{equation}

To prove that the family $\{\tilde S_{\omega_1}\dots \tilde S_{\omega_n}v_0 :\omega_1,\dots,\omega_n\in\{0,\dots,N-1\}\}$ is a Parseval frame, we construct a dilation to an orthonormal basis. We recall two important results in dilation theory:

\begin{theorem}\label{th1.1}
A family $\{\tilde e_j : j\in J\}$ is a Parseval frame for a Hilbert space $\mathcal H$ if and only if there is a larger Hilbert space $\mathcal K\supset \mathcal H$ and an orthonormal basis $\{ e_j : j\in J\}$ such that $P_{\mathcal H} e_j=\tilde e_j$ for all $j\in J$, where $P_{\mathcal H}$ is the orthogonal projection from $\mathcal K$ onto the subspace $\mathcal H$.
\end{theorem}

The second result \cite[Theorem 5.1]{BJKW00}, based on Popescu's dilation theory \cite{Pop89}, shows that the relation \eqref{eqhalf} can always be dilated to a representation of the Cuntz algebra.

\begin{theorem}\label{thj}
Let $\mathcal H$ be a Hilbert space and let $\tilde S_0,\dots, \tilde S_{N-1}$ be operators on $\mathcal H$ satisfying 
$$\sum_{i=0}^{N-1}\tilde S_i\tilde S_i^*=I_{\mathcal H}.$$
Then $\mathcal H$ can be embedded into a larger Hilbert space $\mathcal K$, carrying a representation $ S_0,\dots,  S_{N-1}$ of the Cuntz algebra $\mathcal O_N$ such that, if $P_{\mathcal H}$ is the projection onto $\mathcal H$, we have 
$$\tilde S_i^*= S_i^*P_{\mathcal H},$$
(i.e., $S_i^*\mathcal H\subset \mathcal H$ and $ S_i^*P_{\mathcal H}=P_{\mathcal H} S_i^*P_{\mathcal H}=\tilde S_i^*$) and $\mathcal H$ is cyclic for the representation. The system $$(\mathcal K,  S_0,\dots, S_{N-1}, P_{\mathcal H})$$ is unique up to unitary equivalence.  
\end{theorem}

These are the general lines of our construction. Now we describe the particulars of our construction.

We start with a matrix $T$ of the form 
\begin{equation}\label{matrix_T0}
    T :=\frac{1}{\sqrt{N}} \left( \alpha_{i,j} \right)_{\substack{i=0, \dots , M-1 \\ j=0, \dots , N-1}},
\end{equation}
such that 

\begin{equation}\label{matrix_T}
    T^*T=I_N,
\end{equation}
i.e., an isometry. This means that the columns are orthonormal vectors in $\bc^M$, and, equivalently, that the rows form a Parseval frame for $\bc^N$ (see, e.g., \cite[Lemma 3.8]{HKLW07}).

We assume that
\begin{equation}\label{alpha-zero}
\alpha_{0,j}=1~~\textnormal{for}~j \in \{ 0, \dots , N-1  \}.    
\end{equation}
i.e., the first row of $T$ is $1/\sqrt{N}$. (This is required for the relation $S_0v_0=v_0$).

Next we build the piecewise constant functions
$$m_{i}(x)=\sum_{j=0}^{N-1} \alpha_{i,j} \chi_{[j/{N}~,~{(j+1)}/{N}]},~~ i \in \{ 0, \dots , M-1 \},$$
where $\chi_A$ denotes the characteristic function of the subset $A$.

Using these functions, we define the operators:
\begin{equation}\label{cuntz}
    \left(\Tilde{S}_{i} f \right)(x) := m_{i}(x)f(Nx\mod{1}),~~\textnormal{on}~L^{2}[0,1],
\end{equation}
with $\Tilde{S}_{0}\mathbf{1}=\mathbf 1,$ where $\mathbf{1}$ denote the constant function.

\begin{proposition}\label{pr1.1}
The operators $\tilde S_0,\dots, \tilde S_{M-1}$ satisfy the relation 

\begin{equation}
\sum_{i=0}^{M-1}\tilde S_i\tilde S_i^*=I_{L^2[0,1]}.
\label{eq1.1.10}
\end{equation}

\end{proposition} 

Define $\Omega_\M$ to be the set of all words $\omega_1\dots\omega_{n}$ with digits in $\{0,\dots, M-1\}$, that do not end in $0$, and the empty word $\ty$. (We want the word not to end in 0, to eliminate the repetitions coming from the relation $\tilde S_{\omega_1}\dots\tilde S_{\omega_n}S_0\mathbf 1=\tilde S_{\omega_1}\dots\tilde S_{\omega_n}\mathbf 1$).

\begin{theorem}\label{main-thm} The family of functions
\begin{equation}
    \{ \tilde{S}_{\omega_1}\dots \tilde S_{\omega_n}\mathbf{1} : \omega_1\dots\omega_n \in \Omega_M  \}
\end{equation}
is a Parseval frame in $L^2 [0,1]$.
\end{theorem}

We will start section 2 with the proof of our main result. The proof has the advantage that it shows also how the Parseval frame can be dilated to an orthonormal basis and how the operators $\tilde S_\omega$ are dilated to Cuntz isometries, as in Theorems \ref{th1.1} and \ref{thj}. It has also the advantage that it goes along the more general lines presented in \cite{DPS14, PW17, DR17,DR18}. In Proposition \ref{prop-iterate-parseval-orthogonal}, we present some more properties of the Parseval frames constructed in Theorem \ref{main-thm}, with explicit ways of computing these piecewise constant functions by means of tensor products of matrices. In Remark \ref{rem2.1}, we present a more direct proof of Theorem \ref{main-thm}, without the use of dilation theory. We end the paper with Proposition \ref{pr2.5}, which shows how one can construct examples of matrices satisfying \eqref{matrix_T} and \eqref{alpha-zero}.

The construction of Parseval frames, using operators that satisfy \eqref{eq1.1.10}, is possible in a more general context, but we defer this to a later paper. 

\section{Proofs and other results}

\begin{proof}[Proof of Proposition \ref{pr1.1}]
We compute $\tilde S_l^*$. We have, for $f,g\in L^2[0,1]$,
$$\ip{\tilde S_l f}{g}=\int_{[0,1]}m_l(x)f(Nx\mod 1)\cj g(x)\,dx=\frac{1}{N}\sum_{b=0}^{N-1}\int_{[0,1]}m_l\left(\frac{x+b}{N}\right)f(x)\cj g\left(\frac{x+b}{N}\right)\,dx.$$
Thus 
\begin{equation}
\tilde S_l^* g(x)=\frac{1}{N}\sum_{b=0}^{N-1}\cj m_l\left(\frac{x+b}{N}\right)g\left(\frac{x+b}{N}\right)=\frac{1}{N}\sum_{b=0}^{N-1}\cj \alpha_{l,b}g\left(\frac{x+b}{N}\right).
\label{eq1.1.1}
\end{equation}
So, if $x\in\left[\frac{b'}{N},\frac{b'+1}{N}\right)$, and $g\in L^2[0,1]$, then 
$$\sum_{l=0}^{M-1}\tilde S_l\tilde S_l^*g(x)=\sum_{l=0}^{M-1}m_l(x)\frac{1}{N}\sum_{b=0}^{N-1}\cj\alpha_{l,b}g\left(\frac{(Nx\mod 1)+b}{N}\right)$$$$=
\sum_{b=0}^{N-1}g\left(x+\frac{-b'+b}{N}\right)\frac1N\sum_{l=0}^{M-1}\alpha_{l,b'}\cj\alpha_{l,b}=\sum_{b=0}^{N-1}g\left(x+\frac{-b'+b}{N}\right)\delta_{b,b'}=g(x).$$
\end{proof}

\begin{proof}[Proof of Theorem \ref{main-thm}]

First we will dilate the isometry matrix $T$ to a unitary in a special way. 

Pick a number $N'$ such that $NN'\geq M$. Denote $B=\{ 0 , \dots , N-1 \},~B'=\{ 0, 1, \dots , N'-1  \}$,  and $L = \{ 0 , \dots , M-1 \}$. 

We can identify $L$ with a subset $L'$ of $B \times B'$ by some injective function $\iota:L\rightarrow B\times B'$, in such a way that $0$ from $L$ corresponds to $(0,0)=\iota(0)$ from $B \times B',$ and let $\alpha_{(b,b'), c}=\alpha_{l,c}$ if $(b,b')=\iota(l)$, $\alpha_{(b,b'),c}=0$ if $(b,b')\not\in \iota(L)$. In other words we add some zero rows to the matrix $T=\frac{1}{\sqrt N}(\alpha_{ij})$ to get $NN'$ rows in total. 

The next step consists of dilating the Parseval frame of row vectors for the matrix $T$ to an orthonormal basis, in a way that is compatible with the Cartesian product structure of $B\times B'$.

We construct the numbers $a_{(b,b'),(c,c')}$, $(b,b'),(c,c')\in B\times B'$ with the following properties:
\begin{enumerate}
	\item The matrix 
	\begin{equation}
	\frac{1}{\sqrt{NN'}}\left(a_{(b,b'),(c,c')} \right)_{(b,b'),(c,c')\in B\times B'}
	\label{eqp1.1}
	\end{equation}
	is unitary and the first row is constant $\frac{1}{\sqrt{NN'}}$ so $a_{(0,0),(c,c')}=1$ for all $(c,c')\in B\times B'$,
	\item For all $(b,b')\in B\times B'$, $c\in B$
	\begin{equation}
	\frac{1}{N'}\sum_{c'\in B'}a_{(b,b'),(c,c')}=\alpha_{(b,b'),c}\quad( (b,b')\in B\times B', c\in B).
	\label{eqp1.2}
	\end{equation}
\end{enumerate}
Let $t_{(b,b^{\prime}),c}=\frac{1}{\sqrt{N}} \alpha_{(b,b^{\prime}), c}.$ Note that the vectors $t_{\cdot, c}$, $c\in B$, in $\mathbb{C}^{NN^{\prime}}$ are orthonormal. Therefore, we can complete it to an orthonormal basis in $\bc^{NN'}$, so we can define some vectors $t_{\cdot , d}, \ d \in \{ 1, . . . , NN^{\prime}-N \}$ such that   
$$\{ t_{\cdot , c} : c \in B \} \cup \{ t_{\cdot , d} : d \in \{ 1, . . . , NN^{\prime}-N \} \}$$
is an orthonormal basis for $\mathbb{C}^{NN^{\prime}}.$

For $c \in B$, define the vectors in $\mathbb{C}^{NN^{\prime}}$ by 
$$\tilde e_{c}(c_1, {c^{\prime}_1})=\frac{1}{\sqrt{N^{\prime}}} \delta_{cc_1} \ \ \  \left( (c_1, {c^{\prime}_1}) \in B \times B^{\prime}  \right).$$  

It is easy to see that these vectors are orthonormal in $\mathbb{C}^{NN'}$, therefore we can complete them to an orthonormal basis for $\mathbb{C}^{NN'}$ with some vectors $\tilde e_d~,~ d \in \{ 1, . . . , NN'-N \}$. 

Note that the vectors $\{\tilde e_c : c \in B  \}$ span the subspace
$$\mathcal{M}=\{ \left( X{(c,c^{\prime})} \right)_{(c,c^{\prime}) \in B \times B^{\prime} } : \textnormal{$X$ does not depend on} \ c^{\prime}   \}.$$

Define now 
$$ {s}_{(b,b^{\prime})}=\sum\limits_{c \in B} t_{(b,b^{\prime}), c} \tilde e_{c}+\sum\limits_{d=1}^{NN^{\prime}-N} t_{(b,b^{\prime}), d}\tilde e_{d}.$$
Since the matrix with columns $t_{\cdot, c}$  and $t_{\cdot, d}$ has orthonormal columns, it is unitary. So it has orthogonal rows. So the vectors $t_{(b,b^{\prime}), \cdot}$ are orthonormal, therefore the vectors ${s}_{(b,b')}$ are orthonormal. Also, since  $\alpha_{(0,0),c}=1$,  we have that $t_{(0,0),c}=\frac{1}{\sqrt{N}}$  for all $c \in B.$ But then 
$$\sum\limits_{c \in B} |t_{(0,0),c}|^2 =1=\| t_{(0,0),\cdot} \|^2.$$
So $t_{(0,0),d}=0$   for $d \in \{ 1, . . . , NN'-N \}.$ Therefore, for all $(c_1, {c^{\prime}_1}) \in B \times B':$
$${s}_{(0,0)}(c_1, {c^{\prime}_1})=\sum\limits_{c \in B} \frac{1}{\sqrt{N}} \ \frac{1}{\sqrt{N'}} \delta_{cc_1}=\frac{1}{\sqrt{N N'}}.$$

Since the vectors $\{\tilde e_c : c \in B  \}$ span the subspace
$$\mathcal{M}=\{ \left( X{(c,c^{\prime})} \right)_{(c,c^{\prime}) \in B \times B^{\prime} } : \textnormal{$X$ does not depend on} \ c^{\prime}   \},$$
the vectors $\{\tilde e_d : d \in \{ 1, . . . , NN'-N \} \}$ are orthogonal to $\mathcal{M}$. Let $P_{\mathcal{M}}$ be the projection onto $\mathcal{M}$.

Note that, for $X \in \mathbb{C}^{N N'},$ we have
$$\left( P_{\mathcal{M}}X  \right) (c_1,c_1^{\prime}) =\sum_{c\in B}\langle X , e_c \rangle e_c(c_1,c_1')=\frac{1}{N'}\sum_{c\in B}\left(\sum_{(c_0,c_0')\in B\times B'}X(c_0,c_0')\delta_{cc_0}\right)\delta_{cc_1}$$
$$=\frac{1}{N'} \sum\limits_{c_0^{\prime} \in B'} X(c_1, c_0^{\prime}).$$
Also, since $P_{\mathcal M}\tilde e_c=\tilde e_c$ for $c\in B$ and $P_{\mathcal M}\tilde e_d=0$ for $d=1,\dots,NN'-N$, we have,
$$
\left( P_{\mathcal{M}} s_{(b,b^{\prime})} \right)(c_1, c_1^{\prime}) = \sum\limits_{c \in B} t_{(b,b'), c} e_c (c_1, c_1^{\prime}) 
=\sum_{c \in B} t_{(b,b'), c}\frac1{\sqrt{N'}}\delta_{cc_1} =\frac{1}{\sqrt{N'}} t_{(b,b'), c_1}.$$
Define now
$$a_{(b,b'),(c,c')} :=\sqrt{NN'}~s_{(b,b')} (c,c').$$
Then we have
$$a_{(0,0), (c,c')}=1 ~ \textnormal{for all}~(c,c').$$
The matrix
$$\frac{1}{\sqrt{N N'}} \left( a_{(b,b'), (c,c')} \right)_{(b,b'), (c,c')}$$
is the matrix with rows ${s}_{(b,b')}.$ So it is unitary.

$$\frac{1}{N^{\prime}} \sum\limits_{c' \in B'} a_{(b,b'), (c,c')} =\frac{1}{N'}\sum_{c'\in B'}\sqrt{NN'}~s_{(b,b')}(c,c')=\sqrt{NN'} \left( P_{\mathcal{M}} {s}_{(b,b')} \right){(c,c')} 
$$$$=\sqrt{NN'} \cdot\frac{1}{\sqrt{N'}} t_{(b,b^{\prime}), c} 
=\alpha_{(b,b^{\prime}), c}.$$
Thus, the conditions (i) and (ii) for the numbers $a_{(b,b'),(c,c')}$ are satisfied.

Next, using the unitary matrix in \eqref{eqp1.1}, we construct some Cuntz isometries $S_{(b,b')}$, $(b,b')\in B\times B'$ in the dilation space $L^2([0,1]\times[0,1])$ and with them we construct an orthonormal set, by applying the Cuntz isometries to the constant function $\mathbf 1$.

Define now the maps $\mathcal{R},\mathcal R': [0,1] \rightarrow [0,1]$ by 
\begin{equation}
\mathcal{R}x={N}x\mod 1,\quad \mathcal{R}'x={N}'x\mod 1,
\end{equation}
and define the maps
\begin{equation}
\Upsilon_{(b,b')}(x,x')=\left( N^{-1}(x+b), N'^{-1}(x'+b')\right)
\end{equation}
for $(x,x^{\prime}) \in \mathbb{R}^d \times \mathbb{R}^{d^{\prime}}$ and $(b,b^{\prime}) \in B \times B^{\prime}.$ 
Define the functions
$$m_{(b,b^{\prime})}(x,x^{\prime}) :=\sum\limits_{(c,c^{\prime}) \in B \times B^{\prime}} a_{(b,b^{\prime}),(c,c^{\prime})} \chi_{\Upsilon_{(c,c^{\prime})}([0,1] \times [0,1])}(x,x^{\prime}),$$
where $\chi_A$ denotes the characteristic function of the set $A$.

With these filters we define the operators $S_{(b,b')}$
on $L^{2}([0,1]\times [0,1])$ by
\begin{equation}
\left( S_{(b,b')} f \right)(x,x')=m_{(b,b')}(x,x') f(\mathcal{R}x,\mathcal{R}x').
\end{equation}

\begin{lemma}
The operators $S_{(b,b')}$, $(b,b')\in B\times B'$ are a representation of the Cuntz algebra $\mathcal O_{NN'}$, i.e., they satisfy the relations in \eqref{eqcuntz}. The adjoint $S_{(b,b')}^*$ is given by the formula
\begin{equation}
(S_{(b,b')}^*f)(x,x')=\frac{1}{NN'}\sum_{(c,c')\in B\times B'}\overline{m_{(b,b')}}(\Upsilon_{(c,c')}(x,x'))f(\Upsilon_{(c,c')}(x,x')),
\label{eqadj}
\end{equation}
for $f\in L^2([0,1] \times[0,1] ), (x,x')\in[0,1]\times [0,1]$.
\end{lemma}

\begin{proof}
First, we compute the adjoint, using the invariance equations for the Lebesgue measure under the maps $\Upsilon_{(c,c')}$, i.e.,
$$\int_{[0,1]^2}f(x,x')\,d(x,x')=\frac{1}{NN'}\sum_{(c,c')}\int_{[0,1]^2}f(\Upsilon_{(c,c')}(x,x'))\,d(x,x').$$
We have:
$$\langle S_{(b,b')}f , g \rangle=\int_{[0,1]^2} m_{(b,b')}(x,x')f(\mathcal{R} x,\mathcal{R'} x') \overline{g}(x,x')\,d(x,x')$$$$=\frac{1}{NN'}\sum_{(c,c')}\int_{[0,1]^2} m_{(b,b')}(\Upsilon_{(c,c')}(x,x'))f(x,x') \overline{g}(\Upsilon_{(c,c')}(x,x'))~d(x,x'),$$
and this proves \eqref{eqadj}.

We check the Cuntz relations:

$$S_{(i,i')}^*S_{(j,j')}f(x,x')=\frac{1}{NN'}\sum_{(c,c')}\cj m_{(i,i')}(\Upsilon_{(c,c')}(x,x'))m_{(j,j')}(\Upsilon_{(c,c')}(x,x'))f(\mathcal R\times\mathcal R'(\Upsilon_{(c,c')}(x,x'))$$
$$=\frac{1}{NN'}\sum_{(c,c')}\cj a_{(i,i'),(c,c')} a_{(j,j'),(c,c')} f(x,x')=\delta_{(i,i'),(j,j')} f(x,x'),$$
by \eqref{eqp1.1}. Therefore
$$S_{(i,i'}^*S_{(j,j')}=\delta_{(i,i'),(j,j')}I.$$

Now take $(x,x')\in\Upsilon_{(c_0,c_0')}[0,1]^2$ so $\mathcal Rx=Nx-c_0$, $\mathcal R'x'=N'x'-c_0'$. Then
$$\sum_{(i,i')}S_{(i,i')}S_{(i,i')}^* f(x,x')=\sum_{(i,i')}\frac{1}{NN'}\sum_{(c,c')}\cj m_{(i,i')}(\Upsilon_{(c,c')}(\mathcal Rx,\mathcal R'x'))f(\Upsilon_{(c,c')}(\mathcal Rx,\mathcal R'x'))
$$
$$=\frac{1}{NN'}\sum_{(c,c')}\sum_{(i,i')}a_{(i,i'),(c_0,c_0')}\cj a_{(i,i'),(c,c')}f(\left(x+\frac{c-c_0}{N},x'+\frac{c'-c_0'}{N'}\right)$$
$$=\sum_{(c,c')} \delta_{(c_0,c_0'),(c,c')}f(\left(x+\frac{c_0-c}{N},x'+\frac{c_0'-c'}{N'}\right)= f(x,x').$$
Therefore
$$\sum_{(i,i')}S_{(i,i')}S_{(i,i')}^*=I.$$
\end{proof}

For a word $\omega=(b_1,b_1^{\prime}) . . . (b_k,{b_k}^{\prime})$ we compute
$$(S_{\omega}\mathbf{1})(x,x') =(S_{(b_1,b_1^{\prime})} . . . S_{(b_k,b_k^{\prime})} \mathbf{1})(x,x^{\prime}) 
=S_{(b_1,{b_1}^{\prime})} . . . S_{(b_{k-1},{b_{k-1}^{\prime})}} m_{(b_k,b_k^{\prime})}(x,x^{\prime})$$
$$=S_{(b_1,b_1^{\prime})} \dots S_{(b_{k-2},b_{k-2}^{\prime})} m_{(b_{k-1},b_{k-1}^{\prime})}(x,x^{\prime}) m_{(b_{k},b_k^{\prime})}(\mathcal{R}x,\mathcal{R}^{\prime}x')
=\dots$$$$= m_{(b_{1},b_1^{\prime})}(x,x^{\prime}) m_{(b_{2},b_2^{\prime})}(\mathcal{R}x,\mathcal{R}^{\prime}x^{\prime}) \dots m_{(b_{k-1},b_{k-1}^{\prime})}(\mathcal{R}^{k-1}x,{\mathcal{R}^{\prime}}^{k-1}x^{\prime}) .$$

Next we will need to compute the projection $P_{V}S_{\omega} \mathbf{1},$ onto the subspace $V$ of functions which depend only on the first component, $$V=\{ f(x,y)=g(x) : g \in L^2[0,1] \}.$$ It is easy to see that the projection onto $V$ is given by the formula
$$(P_Vf)(x)=\int_{[0,1]} f(x,x')\,dx',\quad (f\in L^2([0,1]\times[0,1])).$$
Using the invariance equation for the Lebesgue measure under the maps $\tau'_{c'}(x')=(x'+c')/N'$, $c'\in\{0,\dots, N'-1\}$, we have
\begin{equation*}
\begin{split}
\left( P_{V}S_{\omega}\mathbf{1} \right)(x)
&=\int_{[0,1]} m_{(b_{1}, b_{1}')}(x,x')~ \dots  ~ m_{(b_{k},b_k^{\prime})}(\mathcal{R}^{k-1}x,{\mathcal{R}^{\prime}}^{k-1}x^{\prime}) ~ dx' \\
&=\frac{1}{N'} \sum\limits_{c' \in B'} \int_{[0,1]} m_{(b_1,b_1^{\prime})}(x,\tau'_{c'}x^{\prime})~\dots~ m_{(b_{k},b_k^{\prime})}(\mathcal{R}^{k-1}x,{\mathcal{R}^{\prime}}^{k-1}\tau'_{c'}x^{\prime})  dx'.
\end{split}
\end{equation*}
But, by \eqref{eqp1.2},
$$\frac{1}{N'} \sum\limits_{c' \in B'} m_{(b_1,b_1^{\prime})}(x,\tau'_{c'}x')=\frac{1}{N'} \sum\limits_{c' \in B'} a_{(b_1,b_1^{\prime}), (b(x),c^{\prime})}=\alpha_{(b,b'),b(x)},$$
where $b(x)=b$ if $x \in \left[\frac{b}{N} , \frac{b+1}{N} \right).$ 
So
$$\left( P_{V}S_{\omega}\mathbf{1} \right)(x)=\alpha_{(b_1,b_1'), b(x)}\int_{[0,1]}  m_{(b_2,b_2^{\prime})}(\mathcal{R} x,x^{\prime}) \dots m_{(b_k,b_k^{\prime})} (\mathcal{R}^{k-1}x,{\mathcal{R}^{\prime}}^{k-2}x') \  d{\mu'}(x').$$
It now follows by induction that
$$\left( P_{V}S_{\omega}\mathbf{1} \right)(x) = \prod_{j=1}^{k}\alpha_{(b_{j}, b_{j}^{\prime}), b(\mathcal{R}^{j-1} x)}.$$

We can compute that, if $\tilde{\omega}$ is a word over $L=\{ 0, \dots , M-1 \}$, then
\begin{equation}\label{eq2.8}
\begin{split}
\tilde{S}_{\tilde{\omega}}\mathbf{1} (x) &= m_{\tilde{\omega}_1}(x)~m_{\tilde{\omega}_2}(\mathcal{R}x)~ \dots~ m_{\tilde{\omega}_k}(\mathcal{R}^{k-1}x) \\
&=\alpha_{\tilde{\omega}_1  , b(x)}~\alpha_{\tilde{\omega}_2  , b(\mathcal{R}x)} \dots~ \alpha_{\tilde{\omega}_k  , b(\mathcal{R}^{k-1} x)}. 
\end{split}
\end{equation}
So for the word $\omega$ over $B \times B'$, $P_{V} S_{\omega} \mathbf{1} = \tilde{S}_{\tilde{\omega}}\mathbf{1}$, if all the digits $\omega$ are in $\iota(L)$ and $\omega_{j}=\iota(\tilde{\omega}_j)$, and $P_{V} S_{\omega} \mathbf{1}=0$ if at least one of the digits $\omega_j$ is not in $\iota (L)$.

We will prove that 

\begin{equation}
\{ S_{\omega} \mathbf{1} : \omega~\textnormal{is a word over}~B \times B',~\textnormal{not ending in}~(0,0) \}
\label{eqso}
\end{equation}
is an orthonormal basis for $L^2 ([0,1] \times [0,1])$.

It is easy to see that the family is orthonormal: if two words $\omega$ and $\omega'$ differ on the $i$-th position, since the Cuntz isometries $S_{\omega_i}$ and $S_{\omega'_i}$ have orthogonal ranges, it follows that $S_\omega\mathbf 1$ and $S_{\omega'}\mathbf 1$ are orthogonal; if $\omega$ and $\omega'$ do not differ on any position, then one is a prefix of the other, and by completing with zeros at the end and using the fact that $S_{(0,0)}\mathbf 1=\mathbf 1$, one obtains again orthogonality. 

 It remains to prove the completeness. 

Note first that, if 
\begin{equation}\label{define-e}
    e_{(t,t')}(x,x'):=e^{2 \pi i (t , t') \cdot (x , x')},~~((t,t') \in \mathbb{R} \times \mathbb{R},~(x,x') \in [0,1] \times [0,1]),
\end{equation}
then
\begin{equation*}
\begin{split}
\left( S_{(b,b')}^{*} e_{(t,t')} \right)(x,x')&=\frac{1}{N N'} \sum_{(c,c') \in B \times B'} \overline{m}_{(b,b')} (\Upsilon_{(c,c')}(x,x'))~ e_{(t,t')}(\Upsilon_{(c,c')}(x,x')) \\
&=\frac{1}{N N'} 
 \sum_{(c,c') \in B \times B'} \overline{a}_{(b,b'),(c,c')}~e^{2 \pi i \left( \frac{t \cdot (x+c)}{N} + \frac{t' \cdot (x' + c')}{N'} \right)} \\
&=\bigg\{ \frac{1}{N N'} 
 \sum_{(c,c') \in B \times B'} \overline{a}_{(b,b'),(c,c')} e^{2 \pi i \left( \frac{t \cdot c}{N} + \frac{t' \cdot  c'}{N'} \right)} \bigg\}~e_{\left( \frac{t}{N} , \frac{t'}{N'} \right)} (x,x') \\
&=\nu_{(b,b')}(t,t') ~ e_{\left( \frac{t}{N} , \frac{t'}{N'} \right)} (x,x'),
\end{split}
\end{equation*}
where 
\begin{equation}\label{define-nu}
 \nu_{(b,b')}(t,t') =  \frac{1}{N N'} 
 \sum_{(c,c') \in B \times B'} \overline{a}_{(b,b'),(c,c')} e^{2 \pi i \left( \frac{t \cdot c}{N} + \frac{t' \cdot  c'}{N'} \right)}.
\end{equation}

Let $\mathcal{K}$ be the closed span of the family $\{ S_{\omega}\mathbf{1} \}$ in \eqref{eqso}, and let $P_{\mathcal{K}}$ be the orthogonal projection onto $\mathcal{K}$. Let $\Omega$ be the set of words over $B\times B'$ that do not end in zero, including the empty word. Define, for $(t,t') \in \mathbb{R} \times \mathbb{R}$,
\begin{equation}
    h(t,t')=\norm{P_{\mathcal{K}} e_{(t,t')}}^{2}.
\end{equation}
We have
\begin{equation*}
\begin{split}
h(t,t')&=\sum_{\omega \in \Omega} \left| \langle e_{(t,t')} , S_{\omega} \mathbf{1} \rangle \right|^2  \\
&=\sum_{\omega_{1} \in B \times B'} \sum_{\omega \in \Omega} \left| \langle e_{(t,t')} , S_{\omega_{1}} S_{\omega} \mathbf{1} \rangle \right|^2 \\
&=\sum_{\omega_{1} \in B \times B'} \sum_{\omega \in \Omega} \left| \langle S_{\omega_{1}}^{*} e_{(t,t')} , S_{\omega} \mathbf{1} \rangle \right|^2 \\
&=\sum_{\omega_{1} \in B \times B'} \left| \nu_{\omega_1} (t,t') \right|^2 ~ \sum_{\omega \in \Omega} \left|\langle e_{\left( \frac{t}{N}~,~\frac{t'}{N'} \right)} , S_{\omega} \mathbf{1} \rangle \right|^2 \\
&=\sum_{\omega_{1} \in B \times B'} \left| \nu_{\omega_1} (t,t') \right|^2 ~ h\left(\frac{t}{N}~ ,~\frac{t'}{N'} \right).
\end{split}
\end{equation*}
Now, note that
\begin{equation*}
\begin{split}
\sum_{\omega_{1} \in B \times B'} \left| \nu_{\omega_1} (t,t') \right|^2 &=\sum_{(b,b')} \frac{1}{(N N')^2} \sum_{(c, c') \in B \times B'} \sum_{(d, d') \in B \times B'} \overline{a}_{(b, b'), (c,c')} a_{(b, b'), (d,d')}~e^{2 \pi i {(\frac tN,\frac{t'}{N'}) \cdot [(c,c') - (d,d')]}}  \\
&=\frac{1}{(N N')^2} \sum_{(c, c'), (d,d')} e^{2 \pi i {(\frac tN,\frac{t'}{N'}) \cdot [(c,c') - (d,d')]}}~\sum_{(b,b')} \overline{a}_{(b, b'), (c,c')} a_{(b, b'), (d,d')} \\
&=1,
\end{split}
\end{equation*}
which follows from the fact that the matrix $\frac{1}{\sqrt{N N'}}( a_{(b,b'), (c,c')})$ is unitary.

So $h(t,t')=h\left(\frac{t}{N} , \frac{t'}{N'} \right)$. Since $h(t,t')=\norm{P_{\mathcal{K}} e_{(t,t')}}^{2}$, we can easily see that $h$ is continuous on $\mathbb{R}^2$. Also, since $e_{(0,0)}=\mathbf 1 \in \mathcal{K}$, we get that $h(0,0)=1$.

By induction, we have
$$h(t,t')=h\left( \frac{t}{N^{n}} , \frac{t'}{(N')^{n}} \right) \xrightarrow{n \rightarrow \infty} h(0,0)=1.$$
It follows that $h(t,t')$ is the constant $1$, which means that $e_{(t,t')}$ is in $\mathcal K$ for any $(t,t')$. By the {Stone-Weierstrass theorem}, $\mathcal{K}=L^2 \left( [0,1] \times [0,1] \right).$ Thus we have a complete orthonormal basis. Hence $P_{V}S_{\omega}\mathbf{1}$ is a Parseval frame. Eliminating the zeros, according to the statement after \eqref{eq2.8}, we obtain that the functions $\tilde S_{\tilde \omega}\mathbf 1$, with $\tilde\omega\in\Omega_M$ form a Parseval frame for $L^2[0,1]$. 
\end{proof}

\begin{remark}\label{rem2.0}
Our proof shows how the Parseval frame $\{\tilde S_{\tilde \omega}\mathbf 1 : \tilde \omega\in\Omega_M\}$ can be dilated to an orthonormal basis and also how the operators $\tilde S_l$, $l\in\{0,\dots,M-1\}$ can be dilated to a representation of the Cuntz algebra as in Theorem \ref{thj}. We describe here, more precisely, what we mean by this.

Note first that we changed the index set $L=\{0,\dots,M-1\}$ to the index set $B\times B'=\{0,\dots,N-1\}\times\{0,\dots,N'-1\}$ and we embedded $L$ into $B\times B'$ by the map $\iota$, with $\iota(0,0)=0$. We also defined $\alpha_{(b,b'),c}=\alpha_{l,c}$ if $(b,b')=\iota(l)$, $c\in B$, and $\alpha_{(b,b'),c}=0$ otherwise. 

The operators $(S_{(b,b')})_{(b,b')\in B\times B'}$ form a representation of the Cuntz algebra $\mathcal O_{NN'}$ on the Hilbert space  $L^2([0,1]\times[0,1])$. 

Let 
$$V=\{f\in L^2([0,1]\times[0,1]) : f(x,y)=g(x),g \in L^2[0,1] \},$$
which can be identified with $L^2[0,1]$. Define the operators $\tilde S_{(b,b')}$ on $L^2[0,1]$, for $(b,b')\in B\times B'$,
$$\tilde S_{(b,b')}=\left\{\begin{array}{cc}\tilde S_l,&\mbox{ if }(b,b')=\iota(l),\\
0,&\mbox{ otherwise.}\end{array}\right.$$

We prove that 
\begin{equation}
S_{(b,b')}^*P_V=\tilde S_{(b,b')}^*,\quad((b,b')\in B\times B').
\label{eqr2.1.1}
\end{equation}
Using the relation before \eqref{define-nu} and the relation \eqref{eqp1.2}, we see that 
$$\left(S_{(b,b')}^*e_{(t,0)}\right)(x,x')=\frac{1}{NN'}\sum_{(c,c')\in B\times B'}\cj a_{(b,b'),(c,c')} e^{2\pi i \frac{t\cdot c}{N}}=\frac{1}{N}\sum_{c\in B}\cj\alpha_{(b,b'),c}e^{2\pi i\frac{t\cdot c}{N}}=\left(\tilde S_{(b,b')}^*e_t\right) (x).$$
The last equality follows from a similar computation to the one just before \eqref{define-nu}, when $(b,b')\in \iota (L)$, and, if $(b,b')\not\in\iota(L)$, $\alpha_{(b,b'),c}=0$ for all $c\in B$. 

Since the functions $e_{(t,0)}$, $t\in \br$ are dense in $V$, we obtain \eqref{eqr2.1.1}. Also, since the vector $\mathbf 1$ is cyclic for the representation (it generates an orthonormal basis for $L^2([0,1]\times [0,1])$ as we have seen before ), we get that $V$ is also cyclic for this representation. 

Thus, by Theorem \ref{main-thm}, the operators $(S_{(b,b')})$ form a representation of the Cuntz algebra $\mathcal O_{NN'}$ which is the dilation of the operators $(\tilde S_{(b,b')})$ which satisfy the relation 
$$\sum_{(b,b')\in B\times B'}\tilde S_{(b,b')}\tilde S_{(b,b')}^*=I_{L^2[0,1]}.$$

The advantage of enlarging the index set from $L$ to $B\times B'$ is that the dilation has a nice structure of a Cartesian product. The disadvantage is that when we project back to the original space we get some extra zeros. 

If we want to avoid these zeros then we can consider the subspace $\tilde K$ which is spanned by the vectors 
$$\{ S_{(b_1,b_1')}\dots S_{(b_k,b_k')}\mathbf 1 : (b_1,b_1')\dots (b_k,b_k')\in \iota (L), k\in\bn\}.$$
It is easy to see that the space $\tilde K$ is invariant for $S_{\iota (l)}$ and $S_{\iota(l)}^*$, $l\in L$, the relation \eqref{eqr2.1.1} is preserved, when restricted to $\tilde K$, and therefore $\tilde K$, with the restrictions of the operators $S_{\iota(l)}$ to it, are exactly the dilation of the operators $\tilde S_l$ as in Theorem \ref{thj}.
\end{remark}

\begin{proposition}\label{prop-iterate-parseval-orthogonal}
Let $A$ be the matrix $\left( \alpha_{l,b} \right)_{l=0,\dots, M-1 , b =0,\dots,N-1}$ and $\overrightarrow{\alpha}_{l}=(\alpha_{l,0},\alpha_{l,1},\dots,\alpha_{l,N-1})$ be the $l^{th}$ row of $A$. 
\begin{enumerate}
    \item Let $l_{0}, l_{1} , \dots , l_{k-1} \in \{0,\dots,M-1\}$ and let $b_{0}, b_{1} , \dots , b_{k-1} \in \{0,\dots, N-1\}$. Define $l:=l_{0}+M l_{1} + \dots + M^{k-1} l_{k-1}$ and $b:=b_{k-1}+N b_{k-2} + \dots + N^{k-1} b_{0}$. Then, for $x \in \left[ \frac{b}{N^{k}} , \frac{b+1}{N^{k}} \right]$, 
\begin{equation}
\begin{split}
\left(\tilde S_{l_{0}}\tilde S_{l_{1}} \dots \tilde S_{l_{k-1}} \mathbf{1} \right)(x) &=\alpha_{l_{0}, b_{0}} \alpha_{l_{1}, b_{1}} \dots \alpha_{l_{k-1}, b_{k-1}} \\
&=(A^{\otimes k})_{l, b}.
\end{split}
\end{equation}
Here $A^{\otimes k}$ is the tensor product $A\otimes A\otimes \dots\otimes A$, $k$ times. 

\item The $l^{th}$ row of $\frac{1}{\sqrt{N^{k}}} A^{\otimes k}$ is $\frac{1}{\sqrt{N^{k}}} (\overrightarrow{\alpha}_{l_{0}} \otimes \overrightarrow{\alpha}_{l_{1}} \otimes \dots \otimes \overrightarrow{\alpha}_{l_{k-1}})$ 
and these rows form a Parseval frame for $\mathbb{C}^{N^{k}}$.
The family
$$\{ \tilde{S}_{\omega}\mathbf{1} : \omega~\textnormal{is a word over}~\{0,\dots,M-1\}~\textnormal{of length}~\leq k~\textnormal{not ending in}~0  \}$$
coincides with the family
$$\{ \tilde{S}_{\omega}\mathbf{1} : \omega~\textnormal{is a word of length}~=k \}$$
and it forms a Parseval frame for the subspace $\mathcal{F}_{k}$ of $L^2$-functions which are constant on every interval of the form $\left[ \frac{b}{N^k} , \frac{b+1}{N^k} \right]$, for $b \in \{ 0, \dots , N^{k-1} \}$.

\item If $\tilde{\omega} \in \Omega_M=\{ \textnormal{words not ending in}~0 \}$, and $\textnormal{length}(\tilde{\omega})=:|\tilde{\omega}| \geq k+1$, then $\tilde{S}_{\tilde{\omega}}\mathbf{1}$ is orthogonal to the subspace $\mathcal{F}_{k}$.
\end{enumerate}
\end{proposition}

\medskip
\noindent
\begin{proof}
(i)
This follows from \eqref{eq2.8}, so , if $x\in\left[\frac{b_1+Nb_0}{N^2},\frac{b_1+Nb_0+1}{N^2}\right]$, then 
$$\tilde S_{l_0}\tilde S_{l_1}\mathbf1 (x)=\alpha_{l_0,b(x)}\alpha_{l_1,b(\mathcal Rx)}=\alpha_{l_0,b_0}\alpha_{l_1,b_1}.$$
 and the fact that 
$$(A \otimes A)_{l_{0}+M l_{1},~ b_{0}+M b_{1}}=\alpha_{l_{0}, b_{0}} \alpha_{l_{1}, b_{1}},$$
and
$$\left( \overrightarrow{\alpha}_{l_{0}} \otimes \overrightarrow{\alpha}_{l_{1}} \right) (b_{0} + N b_{1})=\overrightarrow{\alpha}_{l_{0}}(b_{0}) \overrightarrow{\alpha}_{l_{1}}(b_1).$$

(ii) Since $\frac{1}{\sqrt{N}}A$ is an isometry, i.e., $\frac{1}{N}A^{*}A=I_{{N}}$, then
\begin{equation*}
    \begin{split}
        \frac{1}{N^{k}} (A^{\otimes k})^{*}(A^{\otimes k}) &= \frac{1}{N^{k}} \underbrace{(A^{*}A \otimes \dots \otimes A^{*}A)}_\text{$k$ terms} \\
        &=\underbrace{I_{\mathbb{C}^{N}} \otimes \dots \otimes I_{\mathbb{C}^{N}}}_\text{$k$ terms} \\
        &=I_{\mathbb{C}^{N^{k}}}.
    \end{split}
\end{equation*}
So the rows of $\frac{1}{N^{k}} A^{\otimes k}$ form a Parseval frame for $\mathbb{C}^{N^{k}}$. The subspace $\mathcal{F}_{k}$ is isometric to $\mathbb{C}^{N^{k}}$ by the map $\psi_{k} : \mathbb{C}^{N^{k}} \rightarrow \mathcal{F}_{k}$, which is defined by
$$\psi_{k}(V_{0} , \dots , V_{N^{k}-1}) :=\sqrt{N^{k}} \sum_{b=0}^{N^{k}-1} V_{b}~ \chi_{\left[ \frac{b}{N^k} , \frac{b+1}{N^k}\right]},$$
and 
$$\psi_{k}\left( \frac{1}{\sqrt{N^{k}}} (\overrightarrow{\alpha}_{l_{0}} \otimes \overrightarrow{\alpha}_{l_{1}} \otimes \dots \otimes \overrightarrow{\alpha}_{l_{k-1}})\right)=\tilde{S}_{l_{0}} \dots \tilde{S}_{l_{k-1}} \mathbf{1}.$$
Therefore, the family 
$$\{ \tilde{S}_{\omega}\mathbf{1} : \omega~\textnormal{is a word over}~\{0,\dots,M-1\}~\textnormal{of length}~ k  \}$$
forms a Parseval frame for $\mathcal F_k$. 

Since $\tilde S_0\mathbf 1=\mathbf 1$, we can see that this family coincides with 
$$\{ \tilde{S}_{\omega}\mathbf{1} : \omega~\textnormal{is a word over}~\{0,\dots,M-1\}~\textnormal{of length}~\leq k~\textnormal{not ending in}~0  \}.$$

(iii) With $\tilde{\omega}$ given as before (so $|\tilde{\omega}| \geq k+1$), we have for $f \in \mathcal{F}_{k}$, by $(ii)$ above, that
\begin{equation}
    \norm{f}^{2} = \sum_{\substack{\omega \in \Omega_M, \\ |\omega| \leq k }} \left| \langle f , \tilde S_{\omega}\mathbf{1} \rangle \right|^2.
\end{equation}
On the other hand, since  $\{ \tilde{S}_{\omega}\mathbf{1} : \omega \in \Omega_M  \}$ is a Parseval frame, it follows that
$$
\norm{f}^{2} = \sum_{\omega \in \Omega_M}    \left| \langle f , \tilde S_{\omega}\mathbf{1} \rangle \right|^2 = \sum_{\substack{\omega \in \Omega_M, \\ |\omega| \leq k  }}    \left| \langle f , \tilde S_{\omega}\mathbf{1} \rangle \right|^2 + \sum_{\substack{\omega \in \Omega_M, \\ |\omega| \geq k+1  }}    \left| \langle f , \tilde{S}_{\omega}\mathbf{1} \rangle \right|^2.
$$
From this, we get
$$\sum_{\substack{\omega \in \Omega_M, \\ |\omega| \geq k+1  }}    \left| \langle f , \tilde{S}_{\omega}\mathbf{1} \rangle \right|^2 = 0.$$
So $\tilde{S}_{\tilde{\omega}}\mathbf{1}$ is orthogonal to $f$ if $|\tilde{\omega}| \geq k+1$.

\end{proof}

\begin{remark}\label{rem2.1}

Now we shall provide an alternative proof for Theorem \ref{main-thm}, using Part (ii) in the Proposition \ref{prop-iterate-parseval-orthogonal} above (which we note that it does not require the proof of Theorem \ref{main-thm}).

By (ii) in Proposition \ref{prop-iterate-parseval-orthogonal}, if $|\tilde{\omega}| = m \geq k+1$, we have that $\{ \tilde{S}_{\omega}\mathbf{1} : \omega \in \Omega_M, |\omega| \leq m \}$ is a Parseval frame for $\mathcal{F}_{m}$. So it follows that
$$
\norm{f}^{2} = \sum_{\substack{\omega \in \Omega_M, \\ |\omega| \leq m  }}    \left| \langle f , \tilde S_{\omega}\mathbf{1} \rangle \right|^2 = \sum_{\substack{\omega \in \Omega_M, \\ |\omega| \leq k  }}    \left| \langle f , \tilde S_{\omega}\mathbf{1} \rangle \right|^2 + \sum_{\substack{\omega \in \Omega_M, \\ m \geq |\omega| \geq k+1  }}    \left| \langle f , \tilde{S}_{\omega}\mathbf{1} \rangle \right|^2.
$$
From this, we get
$$\sum_{\substack{\omega \in \Omega_M, \\ m \geq |\omega| \geq k+1  }}    \left| \langle f , \tilde{S}_{\omega}\mathbf{1} \rangle \right|^2 = 0.$$
So $\tilde{S}_{\tilde{\omega}}\mathbf{1}$ is orthogonal to $f$ if $|\tilde{\omega}| \geq k+1$.

Now take $f\in\mathcal F_k$. We have 
$$\sum_{\tilde\omega\in\Omega_M}|\ip{f}{\tilde S_{\tilde\omega}\mathbf 1}|^2=\sum_{\substack{\tilde\omega\in\Omega_M,}\\{|\omega|\leq k}}|\ip{f}{\tilde S_{\tilde\omega}\mathbf 1}|^2
+\sum_{\substack{\tilde\omega\in\Omega_M,}\\{|\omega|\geq k+1}}|\ip{f}{\tilde S_{\tilde\omega}\mathbf 1}|^2=\|f\|^2+0.$$
But the union of the spaces $\mathcal F_k$, $k\in\bn$ is dense in $L^2[0,1]$, and therefore $\tilde S_{\tilde\omega}\mathbf 1$, $\tilde \omega\in\Omega_M$ do form a Parseval frame for $L^2[0,1]$.

\end{remark}

In the following, we present a way to construct examples of matrices $T$ satisfying \eqref{matrix_T} and \eqref{alpha-zero}. As we mentioned in the introduction, this equivalent to the fact that the rows $\overrightarrow T_l=\frac{1}{\sqrt N}(\alpha_{l,0},\dots,\alpha_{l,N-1})$, $l=0,\dots, M-1$ form a Parseval frame for $\bc^N$ and $\overrightarrow T_0(i)=\frac{1}{\sqrt N}$ for all $i=0,\dots N-1$. 

Let $\langle\ora T_0 \rangle$ be the subspace spanned by $\ora T_0$ and $\langle \ora T_0\rangle^\perp$ be its orthogonal complement in $\bc^N$. 
\begin{proposition}\label{pr2.5}
Let $\ora T_l$, $l=0,\dots,M-1$ be a set of vectors in $\bc^N$ with $\ora T_0(i)=\frac{1}{\sqrt N}$ for all $i=0,\dots,N-1$. The following affirmations are equivalent:
\begin{enumerate}
	\item The vectors $\ora T_l$, $l=0,\dots, M-1$ form a Parseval frame for $\bc^N$. 
	\item $\ora T_l\perp \ora T_0$ for all $l=1,\dots, M-1$ and the vectors $\ora T_l$, $l=1,\dots, M-1$, form a Parseval frame for $\langle \ora T_0\rangle^\perp$. 
\end{enumerate}

\end{proposition}

\begin{proof}
Note first that $\|\ora T_0\|=1$. 

(i)$\Rightarrow$(ii). Since the vectors $\ora T_l$, $l=0,\dots, M-1$ form a Parseval frame, we have 
$$1=\|\ora T_0\|^2=|\ip{\ora T_0}{\ora T_0}|^2+\sum_{l=1}^{M-1}|\ip{\ora T_0}{\ora T_l}|^2=1+\sum_{l=1}^{M-1}|\ip{\ora T_0}{\ora T_l}|^2.$$
This implies that $\ora T_l\perp \ora T_0$ for all $l=1,\dots, M-1$.

Take now a vector $f$ in $\langle \ora T_0\rangle^\perp$. We have 
$$\|f\|^2=|\ip{f}{\ora T_0}|^2+\sum_{l=1}^{M-1}|\ip{f}{\ora T_l}|^2=\sum_{l=1}^{M-1}|\ip{f}{\ora T_l}|^2,$$
so $\ora T_l$, $l=1,\dots, M-1$ form a Parseval frame for $\langle \ora T_0\rangle^\perp$.

(ii)$\Rightarrow$(i). Let $f$ be a vector in $\bc^N$. We can decompose $f$ as $f=\ip{f}{\ora T_0}\ora T_0+f_1$, with $f_1\in \langle \ora T_0\rangle^\perp$. We have 
$$\|f\|^2=|\ip{f}{\ora T_0}|^2+\|f_1\|^2=|\ip{f}{\ora T_0}|^2+\sum_{l=1}^{M-1}|\ip{f_1}{\ora T_l}|^2$$$$=|\ip{f}{\ora T_0}|^2+\sum_{l=1}^{M-1}|\ip{f}{\ora T_l}|^2
=\sum_{l=0}^{M-1}|\ip{f}{\ora T_l}|^2.$$
\end{proof}

With Proposition \ref{pr2.5}, we see that, to construct matrices $T$ as in \eqref{matrix_T} and \eqref{alpha-zero}, we just have to construct a Parseval frame for $\langle \ora T_0\rangle^\perp$. This can be done by picking an isometry $\Psi$ from $\bc^{N-1}$ to $\langle \ora T_0\rangle^\perp$, a Parseval frame $e_l$, $l=1,\dots, M-1$ for $\bc^{N-1}$, and letting $\ora T_l=\Psi(e_l)$, $l=1,\dots, M-1$. 

\medskip
 \noindent {\it Acknowledgments}.  This work was partially supported by a grant from the Simons Foundation (\#228539 to Dorin Dutkay) 

\bibliographystyle{alpha}
\bibliography{eframes}


\end{document}